\documentclass[10pt]{amsart}

\usepackage{txfonts}
\usepackage{mathrsfs}
\newtheorem{theorem}{Theorem}[section]
\newtheorem{proposition}[theorem]{Proposition}
\newtheorem{lemma}[theorem]{Lemma}

\newtheorem{remark}{Remark}[section]
\newtheorem{definition}{Definition}[section]
\newtheorem{mainquestion}{Question}
\newtheorem{maintheorem}{Theorem}

\newcommand{\blanksquare}{\,\,\,$\sqcup\!\!\!\!\sqcap$}

\newcounter{example}
{\par}

\usepackage{amscd}
\usepackage{threeparttable}
\usepackage{amssymb}
\usepackage{amsfonts}
\usepackage{epsfig}
\usepackage{graphicx}
\usepackage{amsmath}
\usepackage[all, 2cell]{xy} \UseAllTwocells \SilentMatrices
\usepackage{indentfirst}
\usepackage{color}

\begin{document}

\title[Generic Hamiltonian Dynamics]{\textbf{Generic Hamiltonian Dynamics}}

\author[M. Bessa]{M. Bessa}
\address{M\'ario Bessa, Departamento de Matem\'atica, Universidade da Beira Interior, Rua Marqu\^es d'\'Avila e Bolama,
  6201-001 Covilh\~a
Portugal.}
\email{bessa@ubi.pt}

\author[C. Ferreira]{C. Ferreira}
\author[J. Rocha]{J. Rocha}
\address{C\'elia Ferreira $\&$ Jorge Rocha, Departamento de Matem\'atica, Universidade do Porto, 
Rua do Campo Alegre, 687, 
4169-007 Porto, Portugal}
\email{celiam@fc.up.pt}
\email{jrocha@fc.up.pt}

\author[P. Varandas]{P. Varandas}
\address{Paulo Varandas, Departamento de Matem\'atica, Universidade Federal da Bahia\\
Av. Ademar de Barros s/n, 40170-110 Salvador, Brazil $\&$ CMUP, University of Porto -Portugal}
\email{paulo.varandas@ufba.br}

\date{\today}

\begin{abstract}
In this paper we contribute to the generic theory of Hamiltonians by proving that there is a $C^2$-residual $\mathcal{R}$ in the set of $C^2$ Hamiltonians on a closed symplectic manifold $M$, such that, for any $H\in\mathcal{R}$, there is a full measure subset of energies $e$ in $H(M)$  such that the Hamiltonian level $(H,e)$ is topologically mixing; moreover these level sets are homoclinic classes.
\end{abstract}

\maketitle

2010 \emph{Mathematics Subject Classification}: Primary 37C20, 37C10; Secondary 37J10.

\emph{Keywords}: Hamiltonian vector field, topological transitivity, topological mixing, pseudo-orbits.

\section{Introduction}

\subsection{Hamiltonians, transitivity and mixing}
The Hamiltonian systems form a fundamental subclass of all dynamical systems generated by differential equations. Their importance follows from the vast range of applications throughout different branches of science. In fact, laws of physics are mostly expressed in terms of differential equations, and a well understand and successful subclass of these differential equations, which leave invariant a symplectic structure, are the Hamiltonian equations (see \cite{Abraham}).

Generic properties of such continuous-time systems, i.e. properties which hold on Baire's second category (or non meagre) sets, are thus of great importance and interest since they give us the typical behavior in an appropriate sense that one could expect from the class of models at hand (cf.~\cite{MM,R0, MBJLD, MBJLD2}). There are, of course, considerable limitations to the amount of information we can extract from a specific system by looking at generic cases.
Nevertheless, it is of great utility to learn that a selected model can be slightly perturbed in order to obtain dynamics we understand in a reasonable way.

The topological transitivity is a global property of a dynamical system. As a motivation for this notion, we may think of a real physical system, where a state is never measured exactly. Thus, instead of points, we should study small open subsets of the phase space and describe how they move in that space. 
If each one of these open subsets meet each other by the action of the system after some time, then we say that the system is \textsl{topologically transitive}.
Equivalently, if we take a compact phase space, we may say that the system has a \textit{dense orbit}.
However, if the open subsets remain inseparable after some time, by the iteration of the system, then we say that the system is \textsl{topologically mixing}.
Obviously, a topologically mixing system is also a topologically transitive system.

There exist a lot of transitive systems, as the \textsl{irrational rotations} of $\mathbb{S}^1$, the \textsl{shift maps} and the \textsl{basic sets} (see ~\cite{S}).
It is also well-known that \textsl{$C^{1+\alpha}$-Anosov Hamiltonian flows} ($\alpha>0$) are ergodic and so transitive. In fact, the same holds for \textsl{$C^{1}$-Anosov Hamiltonian flows} because, by the Poincar\'e recurrence theorem, its  non-wandering set equals the whole manifold (energy level) and by the Anosov closing lemma the periodic orbits are dense in the non-wandering set (\cite{S}). Then, using Smale spectral decomposition (\cite{S}), we get only one piece which is transitive. 
Nevertheless, transitivity is not an open property.

\begin{mainquestion}\label{quest1}
Can the transitivity property be generic? 
\end{mainquestion}

Some authors have been working on this question. 
The first remarkable result on this subject is due to Bonatti and Crovisier, in \cite{BC}. They show that, $C^1$-generically, a $C^1$-conservative diffeomorphism is transitive.
Later, jointly with Arnaud, Bonatti and Crovisier extend this result for $C^1$-symplectic diffeomorphisms  (see \cite{ABC}).
Adapting the techniques used to prove these results to the continuous-time case, one of the authors proved an analogous result for $C^1$-divergence-free vector fields. 
In fact, by a result due to Abdenur et al. (see \cite{AAB}), the first author was able to show that, $C^1$-generically, a divergence-free vector field is topologically mixing (see \cite{B1}). Recently, the results in \cite{BC,ABC} got an upgrading in \cite{AC}. In the direction against the abundance of transitivity (ergodicity), but with a much more exigent smoothness hypothesis, Markus and Meyer proved that generic Hamiltonians are neither integrable nor ergodic (\cite{MM}).

In a first approach and since Hamiltonians (see \S\ref{def} for details) are always equipped with the Hamiltonian function $H\colon M\rightarrow\mathbb{R}$, the energy $\{e\}$ and the energy levels $H^{-1}(\{e\})$, it is not clear what could be the statement of the version of \cite{AC,ABC,B1} for Hamiltonians: Can we expect generic Hamiltonians and any energy? Or generic Hamiltonians and generic energy? Or else any other statement? So, the first difficulty is conceptual, namely how to obtain a properly formulation of the result. 

One of the main results of the present paper (Theorem \ref{cor:density}) is a generalization of \cite{BC,ABC} for Hamiltonians and states that ``most" Hamiltonians have ``most" energy levels indecomposable in the sense that we cannot split the energy level or, in other words, there is some orbit that winds around the whole energy level. Here ``most" Hamiltonians means $C^2$-Baire's second category and ``most" energy levels means full Lebesgue measure set.  This theorem contrasts with well-established results for thinner topologies; indeed, KAM theorem (see \cite{Y}) makes impossible to obtain the same result due to the persistence of invariant tori with positive measure. Furthermore, we prove that for $C^2$-generic Hamiltonians each connected component of the energy level is a homoclinic class (Theorem~\ref{homoclinicorol}) and it is topologically mixing (Theorem~\ref{maintheorem}). An ingredient of the proof of previous results is a connecting lemma for pseudo-orbits for Hamiltonians which is of independent interest and whose sketch of the proof is given in the appendix.

\subsection{Basic definitions and statement of the results}\label{def}
To state precisely the main results  we introduce some definitions. Let $(M^{2d},\omega)$ be a symplectic manifold, where $M=M^{2d}$ ($d\geq2$) is a closed, connected and smooth Riemannian manifold, endowed with a symplectic form $\omega$.
Denote by $C^s(M,\mathbb{R})$ the set of $C^s$-real-valued functions on $M$ and call $H\in C^s(M,\mathbb{R})$ a \textsl{$C^s$-Hamiltonian}, for $s\geq2$\label{hamiltdefin}.
From now on, we set $s=2$. Given a Hamiltonian $H$, we can define the \textsl{Hamiltonian vector field} $X_H$ by: 
\begin{equation}\label{formula}
\omega(X_H(p),u)=\nabla_pH(u), \;\forall u\in T_pM,
\end{equation} 
which generates the Hamiltonian flow $X_H^t$. Observe that $H$ is $C^2$ if and only if $X_H$ is $C^1$ and that, since $H$ is smooth and $M$ is closed, $\emph{Sing}(X_H)\neq\emptyset$, where $\emph{Sing}(X_H)$ stands for the singularities of $X_H$ or, in other words, the critical points of $H$. We denote by $Per(H)$ the set of closed orbits for $X^t_H$,  $\mathscr{O}_H(x)$ the $X^t_H$-orbit of $x$ and $\mathscr{O}_H^+(x)$ the forward $X^t_H$-orbit of $x$.
We say that $\tilde{H}$ is $\epsilon$-$C^2$-close to $H$, for $\epsilon>0$ fixed, if $\|H-\tilde{H}\|_{C^2}<\epsilon$, where $\|H-\tilde{H}\|_{C^2}$ denotes the $C^2$-distance between $H$ and $\tilde{H}$.

A scalar $e\in H(M)\subset \mathbb{R}$ is called an \textsl{energy}\label{energydef} of $H$. 
An \textit{energy hypersurface} $\mathcal{E}_{H,e}$\label{energhyperdef} is a connected component of $H^{-1}(\left\{e\right\})$, called \textsl{energy level set}. The energy level set $H^{-1}(\left\{e\right\})$ is said to be \textsl{reg\-u\-lar} if any energy hypersurface of $H^{-1}(\left\{e\right\})$ is reg\-u\-lar, i.e, does not contain any singularity. In this case, we can also say that the energy $e$ is regular.
Observe that a regular \textit{energy hypersurface} is a $X_H^t$-invariant, compact and $(2d-1)$-dimensional manifold. 

Consider a Hamiltonian $H\in C^2(M,\mathbb{R})$, an energy $e\in H(M)$ and a regular energy hypersurface $\mathcal{E}_{H,e}$.
The triplet $(H,e, \mathcal{E}_{H,e})$ is called a \textsl{Hamiltonian system} and the pair $(H,e)$ is called a \textsl{Hamiltonian level}. If $(H,e)$ is regular then $H^{-1}(\{e\})$ corresponds to the union of a finite number of closed connected components, that is, $H^{-1}(\{e\})=\displaystyle\sqcup_{i=1}^{I_e}\mathcal{E}_{H,e,i}$, for $I_e\in\mathbb{N}$. Fixing a small neighborhood $\mathcal{W}$ of a regular energy hypersurface $\mathcal{E}_{H,e}$, there exist a small neighborhood $\mathcal{U}$ of the Hamiltonian $H$ and $\epsilon>0$ such that, for any $\tilde{H} \in \mathcal{U}$ and for any $\tilde{e} \in (e-\epsilon,e+\epsilon)$, we have $\tilde{H}^{-1}(\{\tilde{e}\})\cap \mathcal{W}=\mathcal{E}_{\tilde{H},\tilde{e}}$. 
The energy hypersurface $\mathcal{E}_{\tilde{H},\tilde{e}}$ is called the \textsl{analytic continuation} of $\mathcal{E}_{H,e}$.

Let   $X_H$  be a Hamiltonian vector field, $x$ a regular point  in $M$ and $e=H(x)$.
Define $\mathcal{N}_x:=N_x\cap T_xH^{-1}(\left\{e\right\})$, where $T_xH^{-1}(\left\{e\right\})=\emph{Ker}\: \nabla H(x)$ is the tangent space to the energy level set. 
Thus, $\mathcal{N}_x$ is a ($dim(M)-2$)-dimensional bundle. The \textit{transversal linear Poincar\'{e} flow} associated to $H$ is given by $\Phi_H^t(x): \mathcal{N}_{x}\rightarrow \mathcal{N}_{X_H^t(x)}$ where $\Phi_H^t(x)\cdot v=\Pi_{X_H^t(x)}\circ D{X_H}^t_x(v)$, where $\Pi_{X_H^t(x)}: T_{X_H^t(x)}M\rightarrow \mathcal{N}_{X_H^t(x)}$ denotes the canonical orthogonal projection. Observe that $\mathcal{N}_x$ is $\Phi^t_{H}(x)$-invariant. Now,  $\mathcal{N}_x$ can be seen as the quotient space $T_x\mathcal{E}_{H,e} / \langle X_H \rangle$, and consider the symplectic form $\tilde{\omega}_{\mathcal{E}_{H,e}}: \mathcal{N}_x \times \mathcal{N}_x \rightarrow \mathbb{R}$ defined by $\tilde{\omega}_{\mathcal{E}_{H,e}}([u],[v])=\omega(u,v)$ for any $u,\, v \in T_x\mathcal{E}_{H,e}$. 
Note that this form is well defined since $\tilde{\omega}_{\mathcal{E}_{H,e}}([X_H],[v])=\omega(X_H,v)=\nabla H(v)=0$, for any $v \in T_x\mathcal{E}_{H,e}$.
It is well-known (see e.g. ~\cite{Abraham}) that, given a regular point $p\in\mathcal{E}_{H,e}$, $\Phi_H^t(p)$ is a linear symplectomorphism for  $\tilde{\omega}_{\mathcal{E}_{H,e}}$. Given a closed orbit $p$ of period $\ell>0$ of a Hamiltonian $H$, the transversal linear Poincar\'{e} flow $\Phi^\ell_H(p)$, is the derivative of the standard Poincar\'e map. We say that $p$ is a \emph{hyperbolic point} if the eigenvalues of $\Phi^\ell_H(p)$ do not intersect the unit circle and we say that $p$ is an \emph{elliptic point} if some of the eigenvalues of $\Phi^\ell_H(p)$ are non-real and of norm equal to one and the remaining ones (if any) have norm different from one. Finally, we say that $p$ is a \emph{totally elliptic point} if all the eigenvalues of $\Phi^\ell_H(p)$ are non-real and of norm equal to one.

Given a hyperbolic closed orbit of a Hamiltonian $H$, with period $\ell$, and $p\in\gamma$.
We define the \textit{stable} and \textit{unstable} manifolds of $\gamma$ by $W_H^{s,u}(\gamma)=\bigcup_{0\leq t\leq \ell}X_H^t(W_H^{s,u}(p))$. The \textsl{homoclinic class} of $\gamma$ is defined by $\mathcal{H}_{\gamma,H}=\overline{W^s_H({\gamma})\pitchfork W^u_H({\gamma})}$, where $\overline{A}$ stands for the closure of the set $A$ and $\pitchfork$  denotes the trans\-versal intersection of manifolds. 

It is well-known that a non-empty homoclinic class is invariant by the flow, has a dense orbit, contains a dense set of closed orbits and is transitive. 
Moreover, the hyperbolic closed orbits are dense in the homoclinic class. If \textbf{all} energy hypersurface of $H^{-1}(\{e\})$ are homoclinic classes, we say that $H^{-1}(\{e\})$ is a homoclinic class.


\begin{definition}\label{topmixhamdef}
A compact energy hypersurface $\mathcal{E}_{H,e}$ is \emph{transitive} (respectively \emph{topologically mixing}) if, for any open and non-empty subsets of $\mathcal{E}_{H,e}$, say $U$ and $V$, there is $\tau\in\mathbb{R}$ such that $X_H^\tau(U)\cap V\neq\emptyset$ (respectively $X_H^t(U)\cap V\neq\emptyset$ for any $t\geq\tau$).
A regular Hamiltonian level $(H,e)$ is \emph{transitive} (respectively \emph{topologically mixing}) if \textbf{each one} of the energy hypersurfaces of $H^{-1}(\{e\})$ is transitive (respectively topologically mixing).
\end{definition}

The following result is an immediate consequence of the fact that the Morse functions are $C^2$-open and dense among $C^2(M,\mathbb{R})$ and Sard's theorem. 
\begin{lemma}\label{lemmaregular}
There is a $C^2$-open and dense subset $\mathcal{O}$ in $C^2(M,\mathbb{R})$ such that, for any $H\in\mathcal{O}$ there exists an open and dense set, full Lebesgue measure subset
$\mathcal{S}(H) \subset H(M)$ of energies such that any energy $e\in\mathcal{S}(H)$ the Hamiltonian level $(H,e)$ is regular.
\end{lemma}

Accordingly with this definition, we now can state the our first result.

\begin{maintheorem}\label{translemma}
There is a residual set $\mathcal{R}$ in $C^2(M,\mathbb{R})$ such that, for any $H_0\in\mathcal{R}$ and energy $e$ in a compact subset $I \subset \mathcal{S}(H_0)$ there is a $C^2$ neighborhood $\mathcal{V}$ of $H_0$ (depending also on $I$)
and a residual subset $\mathcal{R}_{H_0}$ of $\mathcal{V}$ such that for any $H\in \mathcal{R}_{H_0}$ and $e\in I$ the Hamiltonian level $(H,e)$ is transitive.  
\end{maintheorem}

The \emph{quasi-ergodic hypothesis} states the existence of a dense orbit in the connected component of the energy levels of a given Hamiltonian. In  \cite[\S 4.4 and \S4.5]{Y} Yoccoz gives examples, in the $C^\infty$ setting, of Hamiltonians in certain symplectic manifolds which do not satisfy the quasi-ergodic hypothesis. In the opposite direction, Theorem~\ref{translemma} establishes the truthness of the quasi-ergodic hypothesis under $C^2$-generic assumptions.

In fact, we can pick the compact sets above to produce more accurate estimates:

\begin{maintheorem}\label{cor:density}
There exists a residual subset $\mathcal{R}$ in $C^2(M,\mathbb{R})$ such that for any $H\in\mathcal{R}$
the non-transitive energy levels in $H(M)$ have zero Lebesgue measure.
\end{maintheorem}

From Theorem \ref{translemma}, we can derive the following result concerning on the \textsl{homoclinic class }of a hyperbolic closed orbit $\gamma$ of $H$, which is the closure of the set of transversal intersections between the stable and unstable manifolds of all points $p$ in $\gamma$. 

\begin{maintheorem}\label{homoclinicorol}
There is a residual set $\mathcal{R}$ in $C^2(M,\mathbb{R})$ such that, for any $H_0\in\mathcal{R}$ and a generic energy $e$ in $H_0(M)$ there is a $C^2$ neighborhood $\mathcal{V}$ of $H_0$ and a residual subset $\mathcal{R}_{H_0}$ of $\mathcal{V}$ such that for any $H\in \mathcal{R}_{H_0}$ any energy hypersurface of $H^{-1}(\{e\})$ is a homoclinic class.
\end{maintheorem}

Finally, combining Theorems \ref{translemma}, \ref{cor:density} and \ref{homoclinicorol}, we obtain a stronger result  proving that, in this context, the topologically mixing property is generic. 

\begin{maintheorem}\label{maintheorem}
There is a residual set $\mathcal{R}$ in $C^2(M,\mathbb{R})$ such that, for any $H_0\in\mathcal{R}$ and a generic energy $e$  in $H_0(M)$ there is a $C^2$ neighborhood $\mathcal{V}$ of $H_0$ and a residual subset $\mathcal{R}_{H_0}$ of $\mathcal{V}$ such that for any $H\in \mathcal{R}_{H_0}$ the Hamiltonian level $(H,e)$ is topologically mixing. Furthermore, for $C^2$-generic Hamiltonians the non-topologically mixing energy levels have zero Lebesgue measure.
\end{maintheorem}

The main tools to prove the previous results are the generic non-existence of resonances for Hamiltonians and a version for Hamiltonians of the Connecting Lemma for pseudo-orbits developed in \cite{ABC} by Arnaud \emph{et al.}.
To state them, we need the notions of \textsl{resonance relations} and of \textsl{pseudo-orbits}, which we now define. 

Consider $H\in C^2(M,\mathbb{R})$. Given $p\in \emph{Sing}(H)$ we consider the eigenvalues $\{\sigma_1,...,\sigma_{2d}\}$ of  $DX_H(p)$. If $q\in Per(H)$ has period $\ell$ we consider the eigenvalues $\{\sigma_1,...,\sigma_{2d-2}\}$ of  $\Phi_H^\ell(q)$. A resonance relation between $\{\sigma_1,...,\sigma_{n}\}$ ($n=2d$ for singularities and $n=2d-2$ for periodic points) is a relation of the type $\sigma_i=\prod_{j=1}^{n}\sigma_j^{k_j}$, for some $i\in\{1,...,n\}$ and $k_1,...,k_{n}$ natural numbers such that either $k_i\neq1$, or else there exists $j\neq i$ such that $k_j\neq 0$.
Since $\Phi_H^{\ell}(q)$ is symplectic, the following trivial resonance relations holds: $\sigma_i=\sigma_i\prod_{k=1}^{d-2}(\sigma_k\sigma_{d+k})^{\alpha_k}$, for naturals $\alpha_k$.
A resonance relation different from these ones is called a non-trivial resonance relation. 
Robinson proved in \cite{R0} that, $C^2$-generically, there are not non-trivial resonance relations.

\begin{theorem}\cite[Theorem 1]{R0}\label{robinson}
There is a residual $\mathcal{R}$ in $C^2(M,\mathbb{R})$ such that, for any $H\in\mathcal{R}$, any $p\in \emph{Sing}(H)$ and any $q\in Per(H)$ with period $\ell$, the eigenvalues of $DX_H(p)$ and of $\Phi_H^{\ell}(q)$ do not satisfy non-trivial resonance relations.
\end{theorem}

Now, we recall the notion of pseudo-orbit and the Connecting Lemma for Hamiltonians. 

\begin{definition}
Consider a Hamiltonian system $(H,e,\mathcal{E}_{H,e})$ and $\epsilon>0$. 
A sequence $\{x_i\}_{i=0}^n$ on $\mathcal{E}_{H,e}$, with $n\in\mathbb{N}$, is \textsl{an $\epsilon$-pseudo-orbit on $\mathcal{E}_{H,e}$} of length $n$, if $d(X^{1}_H(x_i), x_{i+1})<\epsilon$, for any $i\in\{0,...,n-1\}$, where $d(\cdot,\cdot)$ denotes the distance inherited by the Riemannian structure.
\end{definition}

\begin{remark}\label{remorbit}
For divergence-free vector fields, and so for Hamiltonian vector fields, we have that $\Omega(H|_{\mathcal{E}_{H,e}})=\mathcal{E}_{H,e}$. 
Therefore, any $x,y\in \mathcal{E}_{H,e}$ are connected by an $\epsilon$-pseudo-orbit, for any $\epsilon>0$.
\end{remark}

\textbf{Connecting Lemma:} (for pseudo-orbits of Hamiltonians) \emph{Take $H\in C^2(M,\mathbb{R})$ and a reg\-u\-lar energy $e\in H(M)$, such that the eigenvalues of any closed orbit of $H$ do not satisfy non-trivial resonances. 
Then, for any $C^2$-neighborhood $\mathcal{U}$ of $H$, for any energy hypersurface $\mathcal{E}_{H,e}\subset H^{-1}(\{e\})$ and for any $x,y\in\mathcal{E}_{H,e}$ connected by an $\epsilon$-pseudo-orbit, for $\epsilon>0$, there exist $\tilde{H}\in\mathcal{U}$ and $t>0$ such that $e=\tilde{H}(x)$ and $X^t_{\tilde{H}}(x)=y$ on the analytic continuation $\mathcal{E}_{\tilde{H},e}$ of $\mathcal{E}_{H,e}$.}

\medskip

To prove this result, we have to resume the arguments used by Arnaud et al. \cite{ABC,BC}, and to adapt them to the Hamiltonian setting. 
Besides the perturbation techniques, the core of the proofs is the need to restrict our attention to the energy hypersurface, in order to perturb the Hamiltonian and keep the energy, when analyzing the perturbations and their supports. This strategy was firstly followed by Bonatti and Crovisier for diffeomorphisms (see \cite{BC}).
Later, jointly with Arnaud (see \cite{ABC}), these authors proceeded with this methodology to get the proof of the Connecting Lemma for pseudo-orbits of symplectomorphisms. 
The main novelties in the symplectomorphisms context are the need for the perturbations to be \textsl{symplectic} and the fact that closed orbits can be \textsl{stably elliptic}. 
This means that the symplectomorphisms case cannot be reduced to the one treated in \cite{BC}, where closed orbits are assumed to be hyperbolic.
That is why, in \cite{ABC}, the authors prove this result for symplectomorphisms, by doing the necessary changes. 

For the Hamiltonian case, recall that the transversal linear Poincar\'e flow is, in fact, a symplectomorphism and observe that we are assuming the absence of singularities on the energy hypersurfaces. 
Keeping in mind the strategy described in \cite{ABC}, the novelties in the proof of the \textsl{Connecting Lemma for pseudo-orbits of Hamiltonian} are the statement of adequate definitions and, since the energy hypersurfaces are invariant by the Hamiltonian flow, the need for the pseudo-orbit being completely contained in the same energy hypersurface. 
Hence, we have to ensure the creation of symplectic perturbations without leaving the initial energy hypersurface. 
Recall that the energy hypersurface is indexed to the Hamiltonian.
Thus, it may change when we perturb the Hamiltonian. 
That is why, in the statement of the Connecting Lemma, we want the energy of the points in the pseudo-orbit to be kept constant, even if we $C^2$-perturb the Hamiltonian. 
However, since we are allowed to push along the energy levels (see the closing lemma strategy in Pugh-Robinson's paper \cite[\textsection 9(a)]{PughRob}), the arguments stated in \cite{ABC} can be adapted to the Hamiltonian case.
At the Appendix, we give a brief description of the necessary modifications in \cite{ABC} in order to obtain the proof of the Connecting Lemma.

\section{Proof of the Theorems}\label{mainseclemma}

\subsection{Proof of Theorem \ref{translemma}}

Let $\mathcal{R}_0$ be the residual set given by Robinson's theorem (see Theorem \ref{robinson}) and let $\mathcal{O}$ be  as in Lemma \ref{lemmaregular}. Fix $H_0\in \mathcal{R}:=\mathcal{R}_0\cap\mathcal{O}$ and $I \subset \mathcal{S}(H_0)$ be a compact set. Let $\mathcal{V}$ be a small $C^2$-neighborhood of $H_0$  such that  the Hamiltonian level $(H,e)$ is regular for all $H\in\mathcal{V}$ and $e\in I$. For $H\in\mathcal{V}$ and $e\in I$, let $\mathcal{E}_{H,e}$ denote the analytic continuation of $\mathcal{E}_{H_0,e}$.

Consider a countable basis of open sets on $M$, $\left\{U_{n}\right\}_n$, and, for $m, n \in \mathbb{N}$, define the set $\mathcal{P}_{n,m}$ by the following condition: $H \in \mathcal{P}_{n,m}$ if and only if for any $e\in I$ one has 
\begin{equation}\label{cond}
\left[\cup_{t>0}X_H^t(U_{n}\cap \mathcal{E}_{H,e})\right] \cap (U_{m}\cap\mathcal{E}_{H,e})\neq \emptyset.
\end{equation}
We claim that $\mathcal{P}_{n,m}$ is a $C^2$-open set. In fact let us fix $H\in \mathcal{P}_{n,m}$. For $e\in I$ there are open sets $V_e$ and $\mathcal{U}_e$ such that (\ref{cond}) holds for any $\tilde H\in \mathcal{U}_e$ and every $\tilde e\in V_e$. As the family $\{V_e\}_{e\in I}$ is an open cover of $I$, by compactness, we can chose $e_1$, ..., $e_n$ in $I$ such that $\{V_{e_i}\}_{i=1}^n$ is a finite sub cover of $I$. We define the open set $\mathcal{U}_H:=\cap_{i=1}^n\mathcal{U}_{e_i}$. It is clear that $H\in \mathcal{U}_H\subset \mathcal{P}_{n,m}$, thus $\mathcal{P}_{n,m}$ is open.
Now, we define the residual set
$$
\mathcal{R}_{H_0}:=\mathcal{V}\cap\mathcal{R}\cap\bigcap_{n,m\in\mathbb{N}}\bigl(\mathcal{P}_{n,m}\cup(\overline{\mathcal{P}_{n,m}})^c\bigr),
$$ 
where, given a set $A$, $\bar{A}$ stands for its closure and $A^c$ for its complementary. Observe that, fixed $U_n$ and $U_m$, for any $H$, $C^2$-robustly either $H\in \mathcal{P}_{n,m}$ or $H\in (\overline{\mathcal{P}_{n,m}})^c$. 

It remains to show that $\mathcal{E}_{H,e}$ is transitive for any $ H\in \mathcal{R}_{H_0}$ and any $e\in I$. So, by contradiction, assume that there are $H$ and $e$ such that $\mathcal{E}_{H,e}$ is not transitive. Therefore, there are open sets $A, B\in M$ such that $A\cap\mathcal{E}_{H,e}\not=\emptyset$, $B\cap\mathcal{E}_{H,e}\not=\emptyset$ and:
$$
\left[X_H^t(A\cap \mathcal{E}_{H,e})\right] \cap (B\cap\mathcal{E}_{H,e})= \emptyset,
$$
for all $t\in\mathbb{R}$. Since $\left\{U_{n}\right\}_n$ is a basis of the topology it follows that there are $m,n\in\mathbb{N}$ such that for all $t\in\mathbb{R}$
$$
\left[X_H^t(U_n\cap \mathcal{E}_{H,e})\right] \cap (U_m\cap\mathcal{E}_{H,e})= \emptyset,
$$
i.e. $H\in\mathcal{P}_{n,m}^c\cap \mathcal{R}_{H_0}$. Hence, we have $H\in (\overline{\mathcal{P}_{n,m}})^c$.

Choose $x\in U_{n}\cap \mathcal{E}_{H,e}$ and $y\in U_{m}\cap \mathcal{E}_{H,e}$.
By Remark \ref{remorbit}, $x,y$ are connected by an $\epsilon$-pseudo-orbit in $\mathcal{E}_{H,e}$, for any $\epsilon>0$. 
Moreover, since $H\in\mathcal{R}$, we can apply the Connecting Lemma for pseudo-orbits of Hamiltonians.
So, for any $C^2$-neighborhood $\mathcal{U}$ of $H$, there exists  $\tilde{H}\in\mathcal{U}\cap\mathcal{R} \cap\overline{(\mathcal{P}_{n,m})}\:^c$ such that $e=\tilde{H}(x)$ and there is $T>0$ such that $X^T_{\tilde{H}}(x)=y$ on $\mathcal{E}_{\widetilde{H},{e}}$. 
Then $\tilde{H}\in\mathcal{P}_{n,m}$, which is a contradiction. 
Hence $H\in\mathcal{P}_{n,m}$, for all integers $n$ and $m$.
Therefore, $(H,e)$ is transitive, for any $H\in\mathcal{R}_{H_0}$.

\subsection{Proof of Theorem~\ref{cor:density} }\label{density}

This subsection is devoted to the proof of Theorem~\ref{cor:density} concerning estimates of the measure
of energy levels with transitivity.
Let $\varepsilon>0$ be given and $\mathcal{R}$ be the residual set given by Theorem~\ref{translemma}
(in particular $\mathcal{R}\subset \mathcal O$ given by Lemma~\ref{lemmaregular}).
For any $H_0 \in \mathcal{R}$ take a compact set of energy levels $I_{H_0,\varepsilon} 
\subset \mathcal S(H_0)$ with the property that 
$$
\text{Leb}(H_0(M) \setminus I_{H_0,\varepsilon}) <{\varepsilon}^2 \; \text{Leb}(H_0(M))
$$
(or equivalently 
$\text{Leb}(I_{H_0,\varepsilon}) > ( 1- {\varepsilon}^2) \; \text{Leb}(H_0(M))$).
Thus, by Theorem~\ref{translemma} there exists a $C^2$ neighborhood $\mathcal{V}$ of $H_0$ 
and a residual subset $\mathcal{R}_{H_0}$ of $\mathcal{V}$ (both depending on $H_0$ and $\varepsilon$)
such that for any $H\in \mathcal{R}_{H_0}$ and $e\in I_{H_0,\varepsilon}$ the Hamiltonian level $(H,e)$ is transitive. We may assume without loss of
generality (reducing the neighborhood if necessary) that it holds additionally that the symmetric difference satisfies
$$
\text{Leb} ( H(M) \triangle H_0(M)) 
	\le \varepsilon \, \text{Leb} (H_0(M))
$$ 
for every $H\in \mathcal{V}$.
Consequently, if $H\in \mathcal{R}_{H_0}$ we conclude 
\begin{equation}\label{lebeg}
\frac{ \text{Leb} \big(e\in H(M) \colon (H,e) \text{ is transitive} \big) }{ \text{Leb} ( H(M))  } 
	 \ge \frac{ \text{Leb} \big( I_{H_0,\varepsilon} \big) }{ \text{Leb} ( H(M))  } 
	 > \frac{( 1- {\varepsilon}^2) \; \text{Leb}(H_0(M))}{( 1+ {\varepsilon}) \; \text{Leb}(H_0(M))} 
	 = 1-\varepsilon.
\end{equation}
Observe that any element in the residual subset  
$$
\mathcal{R}_\varepsilon= \bigcup_{H_0 \in \mathcal R} \mathcal{R}_{H_0} 
	\subset C^2(M,\mathbb R)
$$
satisfies (\ref{lebeg}).  Now, we just consider the residual subset $\mathcal R =\bigcap_{n\in\mathbb{N}} \mathcal R_{1/n}$ in $C^2(M,\mathbb{R})$ which clearly satisfies the statement of the theorem.

\subsection{Proof of Theorem \ref{homoclinicorol} }\label{maincproof}
In ~\cite[\S5.1]{ABC} it is proved that for generic symplectomorphisms there is only a single homoclinic class using that generic symplectomorphisms are transitive. Therefore, we could try to adopt the same strategy to obtain Theorem~\ref{homoclinicorol} from Theorem~\ref{translemma}. Unfortunately, the argument fails when we consider the analytic continuation of the hyperbolic periodic point (defining the homoclinic class) which could be inside a \emph{non-transitive} energy level. For this reason we must follow a different approach.

Besides, we could also try to use the lower semicontinuity of the homoclinic class referred also in ~\cite[pp. 1431]{ABC}. Unfortunately, the standard argument, which is typical in robust transitive dynamics, of using twice the connecting lemma to glue different homoclinic classes cannot be done since the perturbation given by the connecting lemma (see proof of Claim 1 below) may originate a Hamiltonian without nice properties like transitivity in a given energy level.

\begin{proof}[Proof of Theorem ~\ref{homoclinicorol}]

Let $\mathcal{KS}\subset C^2(M,\mathbb{R})$ be the $C^2$ residual subset of \emph{Kupka-Smale Hamiltonians} given by ~\cite{R0}, i.e., for all $H\in \mathcal{KS}$ all closed orbits are hyperbolic or elliptic. Note that, by Birkhoff fixed point theorem, the hyperbolic orbits on $M$ are dense (cf. \cite{N} Proposition 3.1, Corollary 3.2 and \S6). Indeed, borrowing Pugh-Robinson's generic Fubini argument (see  \cite[pp. 312]{PughRob}), Newhouse's result can be strengthened to: the generic compact energy hypersurface of the generic $C^2$ Hamiltonian contains a dense set of closed hyperbolic points.  We denote by $\mathcal{N}\subset C^2(M,\mathbb{R})$ this previous $C^2$-generic set of Hamiltonians and by $\mathcal{T}_1$ the residual given by Theorem~\ref{translemma}. We claim that $\mathcal{R}:=\mathcal{KS}\cap \mathcal{N}\cap\mathcal{T}_1$ is a residual subset in the conditions of Theorem~\ref{homoclinicorol}. Given $H_0\in\mathcal{R}$ let $e$ be in the intersection of the open and dense set given by Theorem~\ref{translemma} and the residual set of energies of Newhouse's theorem. Considering $I:=\{e\}$ in Theorem~\ref{translemma}, there is a $C^2$ neighborhood $\mathcal{V}$ of $H_0$ and a residual subset $\mathcal{R}_{H_0}$ of $\mathcal{V}$ (depending also on $I$) such that for any $H\in \mathcal{R}_{H_0}$ we have that:  
\begin{itemize}
\item [(i)]  the Hamiltonian level $(H,e)$ is transitive;
\item [(ii)] the hyperbolic closed orbits are dense in $(H,e)$ and
\item [(iii)] the hyperbolic closed orbits have constant index equal to $d$.
\end{itemize}

The proof of Theorem~\ref{homoclinicorol} rely on the next result which can be understood as a sharper version of Robinson's Kupka-Smale theorem (\cite{R0}) with the presence of transitivity, demanding the intersections of stable/unstable manifolds of hyperbolic closed orbits to be transversal and \emph{non-empty}.

\begin{proposition}\label{p1}
There exists a residual subset $\tilde{\mathcal{R}}_{H_0}$ of $\mathcal{V}$ such that for all $H\in\tilde{\mathcal{R}}_{H_0}$ and any hyperbolic closed orbits $\gamma_1,\gamma_2\in\mathcal{E}_{H,e}$ we have $W^u_H(\gamma_1)\pitchfork W^s_H(\gamma_2)\not=\emptyset$.
\end{proposition}

\begin{proof}
Consider the set $\mathcal{R}_{k,\ell} $ of Hamiltonians $H\in \mathcal{V}$ such that for all hyperbolic closed points $\gamma_1,\gamma_2\in\mathcal{E}_{H,e}$ with period $<k$:
\begin{itemize}
\item [(a)] $W^u_{H,\ell}(\gamma_1)\pitchfork W^s_{H,\ell}(\gamma_2)$ (here $W^u_{H,\ell}(\gamma_1)$ denotes the unstable manifold of $\gamma_1$ with size $\ell$) and  
\item [(b)] $W^u_{H}(\gamma_1)\cap W^s_{H}(\gamma_2)\not=\emptyset$ and intersect transversely at some point.
\end{itemize} 

We claim that $\mathcal{R}_{k,\ell}$ is $C^2$-open and dense in $\mathcal{V}$. The $C^2$-openess is clear from the transversality property, continuity of local stable/unstable manifolds and the constant index property in (iii). 
To prove the $C^2$-denseness we assume by contradiction that there exists an open set $\mathcal{W}\subset \mathcal{V}$ such that $\mathcal{R}_{k,\ell}\cap\mathcal{W}=\emptyset$. Since (a) is a $C^2$-open and dense property we can take $H_1\in \mathcal{W}\cap \mathcal{R}_{H_0}$ and hyperbolic closed points $\gamma_1,\gamma_2\in\mathcal{E}_{H,e}$ of period $<k$ satisfying (a) but not (b). Indeed, we can take a smaller neighborhood $\tilde{\mathcal{W}}\subset \mathcal{W}$ such that every $H_1\in \tilde{\mathcal{W}}\cap \mathcal{R}_{H_0}$  and hyperbolic closed points $\gamma_1,\gamma_2\in\mathcal{E}_{H,e}$ of period $<k$ satisfies (a) but not (b). By property (i) above and the Connecting Lemma for Hamiltonians in \cite[Theorem E]{WX} there exists $H_2\in \tilde{\mathcal{W}}$ such that $W^u_{H_2}(\gamma_1)\cap W^s_{H_2}(\gamma_2)\not=\emptyset$ and by another perturbation we find $H_3\in \tilde{\mathcal{W}}$ such that $W^u_{H_3}(\gamma_1)\cap W^s_{H_3}(\gamma_2)\not=\emptyset$ and transversal at some point which is a contradiction.

Just consider the residual subset 
$\tilde{\mathcal{R}}_{H_0}:=\bigcap_{k\in \mathbb{N}}\bigcap_{\ell\in \mathbb{N}} \mathcal{R}_{k,\ell}$
of $\mathcal{V}$. 
Finally, if $H\in \tilde{\mathcal{R}}_{H_0}$ then for any hyperbolic closed points $\gamma_1,\gamma_2\in\mathcal{E}_{H,e}$ we have $W^u_H(\gamma_1)\pitchfork W^s_H(\gamma_2)\not=\emptyset$.
\end{proof}

In conclusion, by property (ii) the hyperbolic closed orbits are dense in $\mathcal{E}_{H,e}$, thus the closure of all hyperbolic closed orbits in $\mathcal{E}_{H,e}$ is the whole hypersurface $\mathcal{E}_{H,e}$. Then, Proposition~\ref{p1} assure that any two hyperbolic closed orbits in $\mathcal{E}_{H,e}$ are homoclinically related and Theorem~\ref{homoclinicorol} is proved.

\end{proof}
\subsection{Proof of Theorem \ref{maintheorem}}

In this section, we state the proof of two auxiliary results for Hamiltonian systems defined on a $2d$-dimensional symplectic manifold, for $d\geq2$. The first one (Lemma~\ref{pastingHamilt}) is a version of the $C^1$-Pasting Lemma for Hamiltonians. 
Actually, in the Hamiltonian setting, the proof of this result is much more simple. The second one (Lemma \ref{periodlemma}) asserts that, $C^2$-generically, the quotient between the period of two distinct closed orbits of a Hamiltonian is irrational.


\begin{lemma}(Pasting Lemma for Hamiltonians)\label{pastingHamilt}
Fix $H_1 \in C^r(M,\mathbb{R})$, $2\leq r\leq \infty$, and let $K$ be a compact subset of $M$ and $U$ a small neighborhood of $K$. Given $\epsilon>0$, there exists $\delta>0$ such that if $H_2\in C^{s}(M,\mathbb{R})$, for $2\leq s \leq \infty$, is $\delta$-$C^{\ell}$-close to $H_1$ on $U$ (where $\ell:=\min\left\{r,s\right\}$) then there exist $H_3\in C^{s}(M,\mathbb{R})$ and a closed set $V$ such that: (i) $K\subset V\subset U$, 
(ii) $H_3=H_2$ on $V$, (iii) $H_3=H_1$ on $U^c$, and
(iv) $H_3$ is $\epsilon$-$C^{\ell}$-close to $H_1$.
\end{lemma}

\begin{proof}
Consider $\left\{U_1, U_2\right\}$ an open cover of $M$, such that $U_1:=U$ and $U_2$ does not contain $K$. Then, there is a $C^\infty$ smooth partition of unity $\left\{\alpha_1, \alpha_2\right\}$, subordinate to $\left\{U_1, U_2\right\}$, such that $\alpha_i:M\rightarrow \left[0,1\right]$ satisfies $supp(\alpha_i)\subseteq U_i$, for $i=1,2$, and $\alpha_1(x)+\alpha_2(x)=1$, for any $x \in M$.
Letting $V:=U_2^c$ and $H_3:=\alpha_1H_2+(1-\alpha_1)H_1$, it is clear that $K\subset V\subset U$,
$H_3=H_2$ on $V$ (since $\alpha_1(x)=1$ and $\alpha_2(x)=0$, for any $x\in V$) and 
$H_3=H_1$ on $U^c$ (since $\alpha_1(x)=0$ and $\alpha_2(x)=1$, for any $x\in U^c$).
Finally, observe that 
$$
\left\|H_3-H_1\right\|_{C^{\ell}}
	\leq \sum_{k=0}^{\ell}\left(\begin{array}{c}\ell\\k\end{array}\right)\|\alpha_1\|_{C^{\ell}}\left\|H_2-H_1\right\|_{C^{\ell}}\leq\sum_{k=0}^{\ell}\left(\begin{array}{c}\ell\\k\end{array}\right)\|\alpha_1\|_{C^{\ell}}\delta<\epsilon
$$
where we used the general Leibniz rule and that $\delta=\delta(\epsilon,\ell)$ can be chosen to satisfy the right-hand side inequality.

\end{proof}

\begin{lemma}\label{periodlemma}
There is a residual $\mathcal{R}$ in $C^2(M,\mathbb{R})$ such that, for any $H\in\mathcal{R}$, any distinct closed orbits for $H$, $\gamma_1$ and $\gamma_2$, with periods $\ell_1$ and $\ell_2$ (respectivelly), satisfy $\dfrac{\ell_{1}}{\ell_{2}}\in\mathbb{R}\backslash\mathbb{Q}.$
\end{lemma}

\begin{proof}
Fix $n\in\mathbb{N}$. 
By Robinson's results ~\cite{R0}, the set $\mathcal{A}_n$ of $H\in C^2(M,\mathbb{R})$ such that its singularities are hyperbolic and its closed orbits of period smaller than $n$ are hyperbolic or elliptic, is open and dense in $C^2(M,\mathbb{R})$.

Let $\{r_i\}_{i=1}^{\infty}$ be an enumeration of the positive rational numbers, with a fixed order. 
Then consider the open and dense set $\mathcal{B}_n$ of Hamiltonians $H\in\mathcal{A}_n$ such that for any two distinct closed orbits $\gamma_1$ and $\gamma_2$ of period smaller than $n$ the following holds $\frac{\ell_{1}}{\ell_{2}}\notin \{r_i\}_{i=1}^n$.

Now, this proof follows the ideas stated in the proof of \cite[Lemma 2.2]{B1}, but using the version of the Pasting Lemma for Hamiltonians, proved in Lemma \ref{pastingHamilt}. 

Fix $\epsilon>0$ and $H_1\in C^2(M,\mathbb{R})$.
By density of $\mathcal{A}_n$, there is $H_2\in\mathcal{A}_n$, $\epsilon$-$C^2$-close to $H_1$. Since $H_2\in\mathcal{A}_n$ the closed orbits with period less than $n$ of $H_2$ are isolated. 
So, they are only finitely many, say $\{\gamma_i\}_{i=1}^m$, for fixed $m\in\mathbb{N}$.

Given a positive sequence $\{s_i\}_{i=1}^m$, the vector field $X_{\overline{H}_i}=\frac{1}{s_i+1}X_{H_2}$ is also a Hamiltonian vector field, for any $1\leq i\leq m$. Actually, by (\ref{formula}), $X_{\overline{H}_i}$ is associated to the Hamiltonian $\frac{1}{s_i+1}H_2$. 
Observe that if we choose $s_i$ arbitrarily close to $0$ then $\overline{H}_i$ is $\epsilon$-$C^2$-close to $H_2$.

For any $1\leq i\leq m$, consider tubular compact neighborhoods $K_i$ of $\gamma_i$, sufficiently small such that some open neighborhoods $\mathcal{W}_i$ of $K_i$ are pairwise disjoint.
The idea now is to apply, recursively $m$ times, Lemma \ref{pastingHamilt}, in order to define $\tilde{H}_m\in C^2(M,\mathbb{R})$ such that $\ell_{{\tilde{H}_m,\gamma_i}}=(1+s_i)\ell_{{H_2,\gamma_i}}$, for $1\leq i\leq m$. Moreover, choosing $s_i$ sufficiently small $\tilde{H}_m$ can be taken $C^2$-arbitrarily close to $H_2$, hence belonging to $\mathcal{A}_n$.
Up to diminishing, if necessary, $\{s_i\}_{i=1}^m$, we may assume without loss of generality that $\frac{\ell_{\tilde{H}_m,\gamma_i}}{\ell_{\tilde{H}_m,\gamma_j}}\notin \{r_i\}_{i=1}^n$, for $i\neq j$ and also that $\tilde{H}_m\in\mathcal{A}_n$. Thus,  $\tilde{H}_m\in\mathcal{B}_n$. Since $\mathcal{B}_n$ is open and dense in $C^2(M,\mathbb{R})$, for any $n\in\mathbb{N}$, the desired residual subset of $C^2(M,\mathbb{R})$ is given by $\mathcal{R}:=\cap_{n\in\mathbb{N}}\mathcal{B}_n$.
\end{proof}

\begin{proof}[Proof of of Theorem \ref{maintheorem}]
Let $\mathcal{R}_0$ be the residual set given by Lemma \ref{periodlemma} and $\mathcal{R}_3$ be the residual set given by Theorem \ref{homoclinicorol}.
Define $\mathcal{R}:=\mathcal{R}_0\cap\mathcal{R}_3$. Fix $H_0\in\mathcal{R}$. Since $H_0\in\mathcal{R}_3$, by Theorem \ref{homoclinicorol}, for a generic energy $e$ in $H_0(M)$ there is a $C^2$ neighborhood $\mathcal{V}$ of $H_0$ and a residual subset $\mathcal{R}_{H_0}$ of $\mathcal{V}$ such that for any $H\in \mathcal{R}_{H_0}$ any energy hypersurface of $H^{-1}(\{e\})$ is a homoclinic class. So, to conclude the proof of Theorem \ref{maintheorem}, we just have to prove that the Hamiltonian level $(H,e)$ is topologically mixing.

Let $\mathcal{E}_{H,e}$ be an energy hypersurface of $H^{-1}(\{e\})$. Let us prove that $\mathcal{E}_{H,e}$ is topologically mixing, that is, for any open, nonempty subsets $U$ and $V$ of $\mathcal{E}_{H,e}$, there is $\tau\in\mathbb{R}$ such that $X^t_H(U)\cap V\neq\emptyset$, for any $t\geq\tau$.

Since hyperbolic closed orbits are dense in the homoclinic class and the index is constant and equal to $d$, we can find two different hyperbolic closed orbits $\gamma_1$ and $\gamma_2$ of $H$, with period $\ell_{1}$ and $\ell_{2}$, such that $ind(\gamma_1)=ind(\gamma_2)=d$ and $\gamma_1\cap U\neq\emptyset$ and $\gamma_2\cap V\neq\emptyset$. 
Moreover, since $H\in\mathcal{R}_0$, we have that $\dfrac{\ell_{1}}{\ell_{2}}\in\mathbb{R}\backslash\mathbb{Q}$.

Fix $x\in \gamma_1\cap U$, $y\in\gamma_2\cap V$ and $z\in W^u(\gamma_1)\cap W^s(\gamma_2)$.
Thus, there is $\tau_1>0$ such that $X_H^{-(\tau_1+m\ell_{1})}(z) \in W^u(x)$ for every $m\in\mathbb{N}$ and
 $\displaystyle\lim_{m\rightarrow+\infty} X_H^{-(\tau_1+m\ell_{1})}(z)=x$.
Then, there is $t_1>0$ such that $X_H^{-(t_1+m\ell_{1})}(z)\in U$ and, therefore, $z\in X_H^{t_1+m\ell_{1}}(U)$, for every $m\ge N_1$.
Similarly, there is $t_2>0$ and a small $\epsilon>0$ such that $X_H^{t_2+n\ell_{2}+s}(z)\in V$, for every  $\left|s\right|<\epsilon$ and $n\ge N_2$. 
From the transitivity of the future orbits of irrational rotations of the circle (cf. ~\cite[Lemma 2]{AAB}) and the fact that $\dfrac{\ell_{1}}{\ell_{2}}\in\mathbb{R}\backslash\mathbb{Q}$, the set 
$\{m\ell_{1}+n\ell_{2}+s: m\ge N_1, n\ge N_2, \left|s\right|<\epsilon\}$  contains an interval of the form $\left[T,+\infty\right)$, for some $T>0$. 
Hence, for any $t\geq t_1+t_2+T$, there are $m\ge N_1$, $n\ge N_2$ and $\left|s\right|<\epsilon$ such that $t=t_1+t_2+m\ell_{1}+n\ell_{2}+s$. 
Then, $X_H^{t_2+n\ell_{2}+s}(z)\in X_H^t(U)\cap V$, for any $t\geq t_1+t_2+T$. 
So, $\mathcal{E}_{H,e}$ is a topologically mixing energy hypersurface.
Therefore, the Hamiltonian level $(H,e)$ is topologically mixing.
\end{proof}


\section*{ Appendix: Hamiltonian Connecting Lemma for pseudo-orbits}\label{CLPO}

This Appendix is quite technical and a more or less direct application of the arguments in \cite{ABC} with the conceptual adaptations and perturbations required for Hamiltonians. In what follows, and for the sake of completeness, we intend to give a short description of the proof explaining the essential steps and the main differences in our setting. We notice that we do not claim any novelty in the following approach. The ideas follows \cite{ABC} but the tools are, of course, with a Hamiltonian flavor.

As explained in \cite{ABC,BC, B1}, the proof of the \textsl{Connecting Lemma for pseudo-orbits} is divided in three main parts. The first step is to prove that the Connecting Lemma concerns on \textsl{local perturbations}. 
These perturbations motivate the definition of \textsl{perturbation boxes} whose support must be in the interior of small open sets, pairwise disjoint till a sufficiently large number of iterates.
Separately, we need to \textsl{analyze the dynamics near closed orbits with small period} because these orbits are not contained in any perturbation box. Finally, we must analyze the \textsl{global dynamics}, in order to cover any orbit with perturbation flowboxes.

\subsection{Local perturbations}

We proceed to describe the modifications of the local perturbation methods in the Hamiltonian context, namely
the lift perturbation process and selection of tilings adapted to pseudo-orbits.

\subsubsection*{Lift axiom}\label{liftsection}
Fix $p\in Per(H)$ and a small neighborhood $U_p$ of $p$.
By the Darboux Theorem (see, for example, \cite[Theorem 3.2.2]{Abraham}), there is a smooth symplectic change of coordinates $\varphi_p:U_p\rightarrow T_{p}M$, such that $\varphi_p(p)=\vec{0}$. Denote by $N_{p,\delta}$ the ball centered in $\vec{0}$ at the normal fiber at $p$ and with radius $\delta$. For a given $\delta>0$ depending on $p$ we let $f_H: \varphi^{-1}_{p}(N_{p,\delta})\rightarrow \varphi^{-1}_{X^1_H(p)}(N_{X^1_H(p)})$ be the canonical Poincar\'e time-one arrival associated to $H$, in fact, given a regular point $p$, we can chose any $\tau>0$ less than its period, if $p$ is periodic.
In \cite{PughRob}, when proving the closing lemma for Hamiltonians, Pugh and Robinson show that the \textsl{lift axiom} is satisfied for Hamiltonians, and they obtain the \emph{closing} from the \emph{lifting}. In short, \emph{lifting} is a way of pushing the orbit along a given direction by a small Hamiltonian perturbation $C^2$-close to the identity. We point out that we never have to push in the direction of increasing energies, i.e. it is possible to push only along the energy surface.
Furthermore, we recall the key point on the using of the $C^1$ topology of the Hamiltonian vector field: \emph{``...one can lift points $p$ in prescribed directions $v$ with results proportional to the support radius"} (\cite[pp. 266]{PughRob}). Allowing that the perturbations can be done in several flowboxes the proportionality constant can be made arbitrarily close to one.

\subsubsection*{Lift Axiom for Hamiltonians.} (cf.~\cite[\S 9 (a)]{PughRob})
Consider a Hamiltonian $H\in C^2(M,\mathbb{R})$ and let $\mathcal{U}$ be a $C^2$-neighborhood of $H$.
Then there are $0<\epsilon\leq1$ and a continuous function $\delta\colon M\setminus \emph{Sing}(X_H)\rightarrow(0,1)$, both depending on $H$ and on $\mathcal{U}$, such that, for any $p$ and $v\in N_{p,\delta(p)}\cap \varphi_p(H^{-1}(H(p)))$, there exists $\tilde{H}\in\mathcal{U}$ satisfying:
\begin{itemize}
	\item $f^{-1}_H\circ f_{\tilde{H}}(p)=\varphi^{-1}_p(\epsilon v)$;
	\item $supp(X_{\tilde{H}}-X_H)$ is contained in the flowbox $\mathcal{T}=\bigcup_{t\in(0,T)}X_H^t(B_{\|v\|}(p)),$
where $B_{\|v\|}(p)$ is taken in a transversal section of $p$ and $T=T(y)$ is such that $T(p)=1$ and $X_H^{T(y)}(y)\in B_{\|v\|}(X_H(p))$, for any $y\in B_{\|v\|}(p)$.
\end{itemize}

\medskip

\subsubsection*{Tiled sections and perturbation flowboxes}\label{flowboxsection}

Given a symplectic chart $\varphi\colon U\rightarrow\mathbb{R}^{2d}$, we say that the cross-section 
$\mathcal{C}$ to the flow on $\mathcal{E}_{H,e}$ on the chart $(U,\varphi)$ is a \emph{tiled section}
if
$\varphi(\mathcal{C}) \subset \mathbb{R}^{2d}$ is symplectomorphic to the standard cube in $\mathbb{R}^{2d-2}$, tilled by smaller cubes by homotheties and translations (see \cite[Fig. 1]{BC}). 
Write $\mathcal{C}=\displaystyle\cup_{k=1}^m \mathcal{T}_k$, with $m\in\mathbb{N}$, where each 
$\mathcal{T}_k$ is called a \textsl{tile of $\mathcal{C}$}.

\begin{figure}[htb]
     \includegraphics[height=4.4cm]{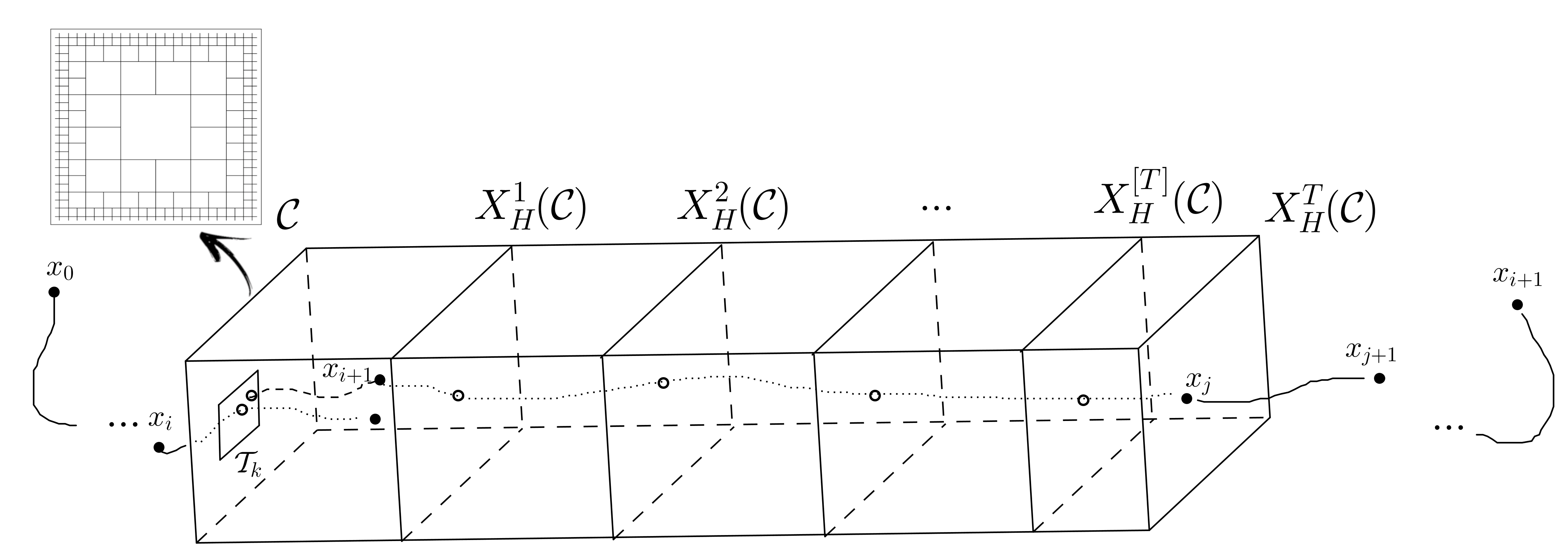}     
\caption{\sffamily Representation of a pseudo-orbit preserving the tiling.}
\end{figure}

\begin{definition}
Consider a Hamiltonian system $(H,e,\mathcal{E}_{H,e})$, a tiled section $\mathcal{C}=\displaystyle\cup_{k=1}^m \mathcal{T}_k$ on $\mathcal{E}_{H,e}$ and a constant $T>0$. 
We say the pseudo-orbit 
$\{x_i\}_{i=0}^n$ on $\mathcal{E}_{H,e}$ ($n\in\mathbb{N}$) \emph{preserves the tiling}
 in the injective flowbox $\mathcal{F}_H(\mathcal{C},T)=\displaystyle\cup_{t\in[0,T]}X^t_H(\mathcal{C})$ 
if $x_0,x_n\notin\mathcal{F}_H(\mathcal{C},T)$ and for any $i\in\{1,...,n-1\}$:
	\begin{itemize}
	\item if $X_H^{[0,1]}(x_i)\cap \mathcal{C}\in\mathcal{T}_k$
	for some $k\in\{1,...,m\}$,  then $X_H^{(-2,0)}(x_{i+1})\cap \mathcal{C}\in\mathcal{T}_k$ and
	\item if $X_H^1(x_i)\in X_H^{[1,T]}(\mathcal{C})$, then $x_{i+1}\in \mathscr{O}_H^+(x_i)$.
\end{itemize}

\end{definition}

This definition implies that the intersection of the pseudo-orbit $\{x_j\}_{j=0}^n$ with the  flowbox $\mathcal{F}_H(\mathcal{C},T)$ is an union of segments of orbits such that the segment $X^{[0,1]}(x_i)$ intersects the tilled cross-section $\mathcal{C}$, and $x_{i+k}$ ($k\geq 1$) belong to the orbit of $x_{i+1}$ while the segment $X^{[0,1]}(x_{i+k})$ intersects the flowbox $\mathcal{F}_H(\mathcal{C},T)$ (see Figure 1).
Moreover, as Pugh and Robinson explained in \cite[\S 9 (a)]{PughRob}, local perturbations on $H$ do not change the energy hypersurfaces in the boundary of the flowboxes where the perturbations take place.
So, we are allowed to push along energy levels. 

The Hayashi Connecting Lemma is a key ingredient to prove the Connecting Lemma for pseudo-orbits of Hamiltonians and, as stated in \cite{WX}, it can be adapted for Hamiltonians. In fact, following ~\cite[Th\'eor\`eme 5]{ABC}, we can extract a slightly stronger statement of the Connecting Lemma for Hamiltonians in \cite[Theorem E]{WX}, which can be seen as a theorem of existence of perturbation flowboxes:

\begin{theorem}\label{perturbprop}
Given a Hamiltonian system $(H,e,\mathcal{E}_{H,e})$ and a $C^2$-neigh\-bor\-hood $\mathcal{U}$ of $H$, there exists $T>0$ such that if $\mathcal{C}$ is a tiled section, then $\mathcal{F}_H(\mathcal{C},T)=\displaystyle\cup_{t\in[0,T]}X^t_H(\mathcal{C})$ is a perturbation flowbox of length $T$, that is:  for any pseudo-orbit $\{x_i\}_{i=0}^n$ on $\mathcal{E}_{H,e}$ preserving the tiling in $\mathcal{F}_H(\mathcal{C},T)$, there exist $\tilde{H}\in\mathcal{U}$ preserving the energy hypersurface, such that $\tilde{H}=H$ restricted to the energy hypersurface and outside $\mathcal{F}_H(\mathcal{C},T)$, and a pseudo-orbit $\{y_j\}_{j=0}^m$ on $\mathcal{E}_{\tilde{H},e}=\mathcal{E}_{H,e}$, with $m\in\mathbb{N}$, such that $y_0=x_0$ and $y_m=x_n$ and the intersection of the pseudo-orbit $\{y_j\}_{j=0}^m$ with $\mathcal{F}_H(\mathcal{C},T)$ is a segment of a true orbit of a point $y_j$ for $X^t_{\tilde{H}}$.
\end{theorem}

For notation simplicity we will call the set $supp(\mathcal{C})=\displaystyle\cup_{t\in[0,T]}X^t_H(closure(\mathcal{C}))$ the  \textsl{support of the perturbation flowbox $\mathcal{C}$} (inside $\mathcal{E}_{H,e}$). 
\subsection{Avoidable closed orbits and covering families}

Notice that the jumps of a pseudo-orbit have no reason to respect the tiling of some perturbation flowbox
and these are not definable for closed orbits with small period.
To deal with this difficulty, we introduce the concept of \textsl{avoidable closed orbits} and of \textsl{covering families}.

\medskip

\noindent\emph{Avoidable closed orbits:}
This kind of orbits are used to derive \textsl{perturbation flowboxes} with disjoint supports, in such a way that the pseudo-orbits stay away from closed orbits with small period.
We anticipate that, if $\mathcal{E}_{H,e}$ has no orbits with small period and all the closed orbits are uniformly avoidable, then we will be able to build a covering family of perturbation flowboxes for $\mathcal{E}_{H,e}$. The next definition is adapted from \cite[Definition 3.10]{ABC}. Consider a Hamiltonian system $(H,e,\mathcal{E}_{H,e})$ and a closed orbit $\gamma$ of $H$ on $\mathcal{E}_{H,e}$. Let $\mathcal{U}$ be a $C^2$-neigh\-bor\-hood of $H$ and fix $T>0$. 
We say a closed orbit $\gamma$ is \textsl{avoidable for ($\mathcal{U}, T$)}, if, for any neighborhood $V_0$ of $\gamma$ and for any $t>0$, there are $\epsilon>0$, open neighborhoods 
$W\subset V\subset V_0$ of the closed orbit $\gamma$ in $\mathcal{E}_{H,e}$, and 
a family of $\mathcal C =\{\mathcal C_i\}_i$  of tiled cross-sections so that the
perturbation flowboxes $\mathcal{F}_H(\mathcal{C}_i,T)=\displaystyle\cup_{t\in[0,T]}X^t_H(\mathcal{C}_i)$
of length $T$ in $\mathcal{E}_{H,e}$ are contained in $V$ and have disjoint supports, and satisfies:
\begin{enumerate}
	\item [(a)] there exist two families of compacts $\mathcal{I}$ and $\mathcal{O}$ contained in the interior of the tiles of $\mathcal{C}$ such that any segment of any $\epsilon$-pseudo-orbit on $\mathcal{E}_{H,e}$ starting outside $V$ (respectively, inside $W$) and ending inside $W$ (respectively, outside $V$) intersects $X^{[0,1]}_H(I)$ for some compact $I\in\mathcal{I}$ (respectively, intersects $X^{[0,1]}_H(O)$ for some compact $O\in\mathcal{O}$);
	\item [(b)] for any compacts $I\in\mathcal{I}$ and $O\in\mathcal{O}$, there exist a pseudo-orbit on $\mathcal{E}_{H,e}$,  with jumps inside $X^{[0,1]}_H(\mathcal{C})$ and preserving the tile of $\mathcal{C}$, starting in $I$ and ending in $O$;
	\item [(c)] for any $x$ in the closure of $\mathcal{F}_H(\mathcal{C},T)$, the first return time of $X^T_H(x)$ to the closure of the perturbation flowboxes  is larger than $t$.
\end{enumerate}

In a few words, a closed orbit $\gamma$ is avoidable for $(\mathcal{U}, T)$ if, for any $t>0$, there exists a family of perturbation flowboxes of length $T$ (with tiled cross
sections) such that, given a pseudo-orbit with starting and ending points far from $\gamma$, but passing very close of $\gamma$, we can exchange the segments of the pseudo-orbit passing close of $\gamma$ by segments of another pseudo-orbit with jumps inside the tiles replacing the original pseudo-orbit by another with smaller number of elements. By Theorem \ref{robinson} the closed orbits of a $C^2$-generic Hamiltonian are uniformly avoidable.

The closed orbits for the Hamiltonian $H$ on $\mathcal{E}_{H,e}$ are called \textit{uniformly avoidable}
if they are isolated and there is a $C^2$-neighborhood $\mathcal{U}$ of $H$ and $T>0$ so that all closed orbits in $\mathcal{E}_{H,e}$ are avoidable for ($\mathcal{U}, T$).

\medskip

\emph{Covering families:}\label{covsection} Given a Hamiltonian system $(H,e,\mathcal{E}_{H,e})$, we want to cover the orbits on $\mathcal{E}_{H,e}$ by a family of perturbation flowboxes, with pairwise disjoint supports. In general, if $\mathcal{E}_{H,e}$ contains closed orbits with small period, then $\mathcal{E}_{H,e}$ cannot have a covering family.
In fact, this kind of closed orbits are disjoint from the perturbation flowboxes.
This motivates the definition of \textsl{covering families outside $\mathcal{V}=\cup_{j=1}^r V_j$}, where the sets $V_j$ ($1\leq j\leq r$) are, in fact, neighborhoods of these closed orbits with small period.

Let $\mathcal{U}$ be a $C^2$-neighborhood of $H$ and let $(\mathcal{F}_H(\mathcal{C}_i,T))_i$ denote a family of perturbation flowboxes for 
($H,\mathcal{U}$), with pairwise disjoint supports, and $\mathcal{V}$ denote a finite family of non-empty open subsets of $\mathcal{E}_{H,e}$ with pairwise disjoint supports. We say that a family $(\mathcal{F}_H(\mathcal{C}_i,T))_i$ of perturbation flowboxes for ($H,\mathcal{U}$) with disjoint supports is a \emph{covering family of 
$\mathcal{E}_{H,e}$} if  there exists a family of compact subsets $\mathcal D \subset \bigcup_i interior(\mathcal{C}_i)$ and 
$t>0$ so that any orbit segment of $x\in\mathcal{E}_{H,e}$ of length $\geq t$ intersects some element in $\mathcal D$. The following definition is an adaption of \cite[Definition 3.2]{ABC} for Hamiltonians. A finite set of perturbation flowboxes $(\mathcal{F}_H(\mathcal{C}_i,T))_i$ for $(H,\mathcal{U})$ and with pairwise disjoint supports is said to be a \emph{covering family of $\mathcal{E}_{H,e}$ outside $\mathcal{V}$} if there are 
\begin{itemize}
\item $t>0$ and $\epsilon>0$;
	\item an open set $W_j$ and a compact set $F_j$, such that $F_j\subset W_j\subset V_j$, for every $j\in\{1,...,r\}$;
	\item a finite family of compacts $\mathcal{D}=\displaystyle\cup_{i=1}^s D_i$ on $\mathcal{C}$, such that every $D_i$ is contained in the interior of a tile of $\mathcal{C}=\bigcup_i \mathcal{C}_i$;
	\item two families $\mathcal{I}_{j}$ and $\mathcal{O}_{j}$ contained in $\mathcal{D}$ such that the support of the flowboxes of the tiles of $\mathcal{C}$ containing this compacts is contained in $V_j$, for any $j\in\{1,...,r\}$,
	\end{itemize}
such that any segment of any $\epsilon$-pseudo-orbit on $\mathcal{E}_{H,e}$:
\begin{enumerate}
	\item[(a)] with length $\geq t$ intersects some $F_j$ or a compact of $\mathcal{D}$;
	\item[(b)] starting outside $V_j$ and ending inside $W_j$ intersects a compact of $\mathcal{I}_{j}$;
	\item[(c)] starting inside $W_j$ and ending outside $V_j$ intersects a compact of $\mathcal{O}_j$;
\end{enumerate}
and for any $j\in\{1,...,r\}$ and for any compact sets $I\subset\mathcal{I}_{j}$ and $O\subset\mathcal{O}_{j}$, there exists a pseudo-orbit with jumps inside the tiles of $\mathcal{C}$, with starting point in $I$ and ending point in $O$. Roughly speaking, given a \textsl{covering family of $\mathcal{E}_{H,e}$ outside $\mathcal{V}$}, any pseudo-orbit either returns regularly to the tiled cross sections, during the time it passes out of $\mathcal{V}$ or else, intersects a compact set $F_j\subset V_j$.
In this last situation, the pseudo-orbit must go through an ``in set" $I\subset\mathcal{I}_j$ and then through an ``out set" $O\subset\mathcal{O}_j$.
Moreover, we can even switch the segment of the pseudo-orbit between $I$ and $O$ by a pseudo-orbit with jumps inside the tiles of $\mathcal{C}$. The existence of these objects follows, up to considering cross sections, from ~\cite[\S4]{ABC} for symplectomorphisms.

\subsection{Connecting pseudo-orbits} 

Arnaud et al. proved, in \cite[Proposition 4.2]{ABC}, that if the eigenvalues of any closed orbit of a symplectomorphism do not satisfy non-trivial resonance relations, then the closed orbits are uniformly avoidable.
Therefore, since the transversal linear Poincar\'e flow is a symplectomorphism, then for any $H\in C^2(M,\mathbb{R})$ and any periodic point $p$ of $H$ with period $\ell$, the eigenvalues of $\Phi_H^{\ell}(p)$ do not satisfy non-trivial resonances, then the closed orbits of $H$ are \textsl{uniformly avoidable}.

Hence, to prove the Connecting Lemma for pseudo-orbits of Hamiltonians it is enough to show that if $Per(H)$ on $\mathcal{E}_{H,e}$ are uniformly avoidable, then, for any $C^2$-neighborhood $\mathcal{U}$ of $H$ and for any $x,y\in\mathcal{E}_{H,e}$, there is $\tilde{H}\in\mathcal{U}$ and $t>0$, such that $\tilde{H}(x)=e$ and $X^t_{\tilde{H}}(x)=y$. It is obvious that this statement follows immediately if $y\in\mathscr{O}_H(x)$.
In fact, to prove the Connecting Lemma, it is enough to show it for certain points $x,y\in\mathcal{E}_{H,e}$. Indeed, the same argument as the ones in \cite[Lemma 3.12]{ABC} allows us to reduce the proof to the case when $x,y$ are not closed orbits.
So, we assume that $x,y\notin Per(H)$ and $y\notin\mathscr{O}_H(x)$ for every $H$ in a  $C^2$-neighborhood $\mathcal{U}_0$. Using Kakutani towers for flows (\cite{AK,Bessa}) and adapting the ideas  of the proof in \cite[Proposition 3.13]{ABC} we obtain that there exist a neighborhood $\mathcal{U}\subset\mathcal{U}_0$ of $H$, a family of disjoint open sets $\mathcal{V}$ and a family of perturbation flowboxes $(\mathcal{F}_H(\mathcal{C}_i,T))_i$ for $(H,\mathcal{U})$ with disjoint supports, both $\mathcal{V}$ and $(\mathcal{F}_H(\mathcal{C}_i,T))_i$ not containing $x$ nor $y$, such that the perturbation flowboxes covering $\mathcal{E}_{H,e}$ outside $\mathcal{V}$. Now, the classical strategy developed in ~\cite{BC} and carried out in the proof of \cite[Proposition 3.4]{ABC} for symplectomorphisms allows us to obtain a pseudo-orbit preserving the tiling and, consequently one erases, flowbox after flowbox, all the jumps of the pseudo-orbit. Thus, there exist $\tilde{H}\in\mathcal{U}$ and $t>0$, such that $\tilde{H}(x)=e$ and $X^t_{\tilde{H}}(x)=y$.

\vspace{0.2cm}

\noindent\textbf{Acknowledgements:}
MB was partially supported by National Funds through FCT (Fun\-da\c{c}\~{a}o para a Ci\^{e}ncia e a Tecnologia) project PEst-OE/MAT/UI0212/2011. CF was supported by FCT - Funda\c{c}\~ao para a Ci\^encia e a Tecnologia SFRH/BD/33100/2007.  
JR was partially supported by FCT - Funda\c{c}\~ao para a Ci\^encia e a Tecnologia through the project 
CMUP: PTDC/MAT/099493/2008.
PV was partially supported by a CNPq-Brazil postdoctoral fellowship at University of Porto.



\begin{thebibliography}{ABC}



\bibitem{AAB} {F. Abdenur, A. Avila and J. Bochi}, \emph{Robust transitivity and topological mixing for $C^1$-flows.} Proc. Amer. Math. Soc., \textbf{132}, \textbf{3} (2003) 699--705
\bibitem{AC}  F. Abdenur and S. Crovisier, \emph{Transitivity and topological mixing for $C^1$ diffeomorphisms}. 
Essays in Mathematics and its Applications (2012), 1--16, Springer Berlin Heidelberg
\bibitem{Abraham} {R. Abraham and J.E. Marsden}, \emph{Foundations of Mechanics}. The Benjamin/Cummings Publishing Company. Advanced Book Program, 2$^{nd}$ edition, 1980
\bibitem{AK} {W. Ambrose and S. Kakutani}, \emph{Structure and continuity of measurable flows.}  Duke Math. J., \textbf{9} (1942) 25--42
\bibitem{ABC} {M.-C. Arnaud, C. Bonatti and S. Crovisier}, \emph{Dynamique sympletiques g\'en\'eriques}. Erg. Th. \& Dyn. Syst. \textbf{25}, \textbf{5} (2005) 1401--1436
\bibitem{B1} { M. Bessa}, \emph{A generic incompressible flow is topological mixing.} C. R. Acad. Sci. Paris, Ser. I, \textbf{346} (2008) 1169--1174 
\bibitem{Bessa} { M. Bessa}, \emph{The Lyapunov exponents of generic zero divergence-free three-dimensional vector fields.}  Erg. Th. \& Dyn. Syst., \textbf{27}, \textbf{6} (2007) 1445--1472
\bibitem{MBJLD} { M. Bessa and J. Lopes Dias}, \emph{Generic Dynamics of 4-Dimensional $C^2$ Hamiltonian Systems.}
Commun. in Math. Phys., \textbf{281} (2008) 597--619
\bibitem{MBJLD2} { M. Bessa and J. Lopes Dias}, \emph{Hamiltonian elliptic dynamics on symplectic 4-manifolds.}
Proc. Amer. Math. Soc., \textbf{137} (2009) 585--592 
\bibitem{BC} {C. Bonatti and S. Crovisier}, \emph{R\'ecurrence et g\'en\'ericit\'e}. Invent. Math. \textbf{158}, \textbf{1} (2004) 33--104 
\bibitem{MM} L. Markus and K.R. Meyer, \emph{Generic Hamiltonian dynamical systems are neither integrable nor ergodic}. Memoirs of the American Mathematical Society, 144. American Mathematical Society, Providence, R.I., 1974
\bibitem{N} S. Newhouse, {Quasi-elliptic periodic points in conservative dynamical systems}. Amer. J. Math. 99, 5 (1977), 1061--1087
\bibitem{PughRob} { C. Pugh and C. Robinson},  \emph{The $C^1$ closing lemma, including Hamiltonians.} Ergod. Th \& Dynam. Sys. \textbf{3} (1983) 261--313
\bibitem{R0} {C. Robinson}, \emph{Generic properties of conservative systems I and II.} Amer. J. Math., \textbf{92} (1970) 562--603
\bibitem{S} M. Shub, \emph{Global stability of dynamical systems.} Springer-Verlag, New York, 1987
\bibitem{Y} J. C. Yoccoz, \emph{Travaux de Herman sur les Tores invariants}. Ast\'erisque, 206 Exp: 754, 4 (1992), 311--344
\bibitem{WX} {L. Wen and Z. Xia}, \emph{$C^1$ connecting lemmas}. Trans. Amer. Math. Soc., \textbf{352} (2000) 5213--5230

\end{thebibliography}
\end{document}